\documentclass[twocolumn]{autart}    

\usepackage{epsfig} 
\usepackage{mathptmx} 
\usepackage{times} 
\usepackage{amsmath} 
\usepackage{amssymb}  
\usepackage{subfigure}
\usepackage{color}
\usepackage{flushend}
\usepackage{eurosym}
\usepackage{amsmath,theorem}
\usepackage{amssymb}
\usepackage{amsfonts}
\usepackage{psfrag}
\usepackage{array}
\usepackage{shortvrb}
\usepackage{epsf}
\usepackage{graphicx}
\usepackage{rotating}
\usepackage{cite,url}
\usepackage{wasysym}
\usepackage[usenames]{xcolor}

\usepackage{algorithm}
\usepackage{algpseudocode}
\usepackage{pifont}

\makeatletter
\def\BState{\State\hskip-\ALG@thistlm}
\makeatother

\usepackage{graphicx}
\usepackage{eurosym}
\usepackage{amssymb}
\usepackage{amsmath,mathtools}
\usepackage{amsfonts}
\usepackage{epstopdf}
\usepackage{epsf,subfigure}
\usepackage{psfrag}
\usepackage{graphics}
\usepackage{color} 

\newtheorem{proposition}{Proposition}
\newtheorem{definition}{Definition}
\newtheorem{theorem}{Theorem}

\newtheorem{lemma}{Lemma}
\newtheorem{remark}{Remark}
\newtheorem{corollary}{Corollary}

\newenvironment{proof}[1][Proof]{\begin{trivlist}
\item[\hskip \labelsep {\bfseries #1}]}{\end{trivlist}}

\begin{document}

\begin{frontmatter}

\title{A Multiple-Comparison-Systems Method for Distributed Stability Analysis of Large-Scale Nonlinear Systems}


\author[SK]{Soumya Kundu}\ead{soumya@lanl.gov},    
\author[MA]{Marian Anghel}\ead{manghel@lanl.gov}               

\address[SK]{Information Sciences Group (CCS-3) and Center for Nonlinear Studies, Los Alamos National Laboratory, Los Alamos, USA}  
\address[MA]{Information Sciences Group (CCS-3), Los Alamos National Laboratory, Los Alamos, USA}             


\begin{abstract}                          
Lyapunov functions provide a tool to analyze the stability of nonlinear systems without extensively solving the dynamics. Recent advances in sum-of-squares methods have enabled the algorithmic computation of Lyapunov functions for polynomial systems. However, for general large-scale nonlinear networks it is yet very difficult, and often impossible, both computationally and analytically, to find Lyapunov functions. In such cases, a system decomposition coupled to a vector Lyapunov functions approach provides a feasible alternative by analyzing the stability of the nonlinear network through a reduced-order comparison system. However, finding such a comparison system is not trivial and often, for a nonlinear network, there does not exist a single comparison system. In this work, we propose a multiple comparison systems approach for the algorithmic stability analysis of nonlinear systems. Using sum-of-squares methods we design a scalable and distributed algorithm which enables the computation of  comparison systems using only communications between the neighboring subsystems. We demonstrate the algorithm by applying it to an arbitrarily generated network of interacting Van der Pol oscillators.
\end{abstract}

\end{frontmatter}

\section{Introduction}
A key to maintaining the successful operation of real-world engineering systems is to analyze the stability of the systems under disturbances. Lyapunov functions methods provide powerful tools to directly certify stability under disturbances, without solving the complex nonlinear dynamical equations \cite{Lyapunov:1892,Haddad:2008}. However for a general nonlinear system, there is no universal expression for Lyapunov functions. Recent advances in sum-of-squares (SOS) methods and semi-definite programming (SDP), \cite{sostools13,Antonis:2005a,Sturm:1999}, have enabled the algorithmic construction of polynomial Lyapunov functions for nonlinear systems that can be expressed as a set of polynomial differential algebraic equations \cite{PPC:2005, Chesi:2011}. Unfortunately, such computational methods suffer from scalability issues and, in general, become intractable as the system size grows
 \cite{Anderson:2011}. 
 { For this reason more tractable alternatives to SOS optimization have been proposed. 
 One such approach, known as DSOS and SDSOS optimization,  is significantly more scalable since it relies on linear programming and second order cone programming~\cite{Ahmadi:2014}.
 A different approach chooses Lyapunov functions with a chordal graphical structure in order to convert the semidefinite constraints into an equivalent set of smaller semidefinite constraints which can be exploited to solve the SDP programs more efficiently~ \cite{Mason:2014}. Nevertheless, the increased scalability decreases performance since both approximations are usually more conservative than SOS approaches. 
 
Despite these computational advances, global analysis of large-scale systems remains problematic when computational and communication costs are considered. Often, a decomposition-aggregation approach offers a scalable distributed  computing framework,  together with a  flexible analysis of structural perturbations~\cite{Siljak:1978} and decentralized control 
designs~\cite{Siljak:1991}, as required by the locality of perturbations. }
Thus, for large-scale systems,  it is often useful to model the system as a network of small interacting subsystems and study the stability of the full interconnected system with the help of the Lyapunov functions of the isolated subsystems. For example, one approach is to construct a scalar Lyapunov function expressed as a weighted sum of the subsystem Lyapunov functions and use it to certify stability of the full system \cite{Siljak:1972,Weissenberger:1973, Michel:1983,Araki:1978}. However, such a method requires centralized computations and does not scale well with the size of the network. Alternatively, methods based on vector Lyapunov functions, \cite{Bellman:1962,Bailey:1966}, are computationally very attractive due to their parallel structure and scalability{, and have generated considerable interest in recent times \cite{Karafyllis:2015,Kundu:2015ACC,Kundu:2015ECC,Xu:2016}. However, applicability of these methods to large-scale nonlinear systems with guaranteed rate of convergence still remain to be explored, for example \cite{Kundu:2015ACC,Kundu:2015ECC} consider asymptotic stability while the works in \cite{Karafyllis:2015,Xu:2016} are demonstrated on small-scale systems.} 

Inspired by the results on comparison systems, \cite{Conti:1956,Brauer:1961,Beckenbach:1961}, it has been observed that the problem of stability analysis of an interconnected nonlinear system can be reduced to the stability analysis of a linear dynamical system (or, `single comparison system') whose state space consists of the subsystem Lyapunov functions. Success of finding such stable linear comparison system then guarantees exponential stability of the full interconnected nonlinear system. However, for a given interconnected system, computing these comparison systems still remained a challenge. In absence of suitable computational tools, analytical insights were used to build those comparison systems, such as  trigonometric inequalities in power systems networks \cite{Jocic:1978}. {In a recent work \cite{Kundu:2015CDC}, SOS-based direct methods were used to compute the single comparison system for generic nonlinear polynomial systems, with some performance improvements over the traditional methods. However there are major challenges before such a method can be used in large-scale systems. For example, it is generally difficult to construct a single comparison system that can guarantee stability under a wide set of disturbances. Also, while \cite{Kundu:2015CDC} presents a decentralized analysis where the computational burden is shared between the subsystems, the scalability of the analysis is largely dependent on the cumulative size of the neighboring subsystems. 

In this article we present a novel conceptual and computational framework which generalizes the single comparison system approach into a sequence of stable comparison systems, that collectively ascertain stability, while also offering better scalability by parallelizing the subsystem-level SOS problems. The set of multiple comparison systems are to be constructed adaptively in real-time, after a disturbance has occurred.} With the help of SOS and semi-definite programming methods, we develop a fully distributed, parallel and scalable algorithm that enables computation of the comparison systems under a disturbance, with only minimal communication between the immediate neighbors. While this approach is applicable to any generic dynamical system, we choose an arbitrarily generated network of modified\footnote{Parameters are chosen to make the equilibrium point stable.} Van der Pol oscillators \cite{van:1926} for illustration. Under a disturbance, the subsystems communicate with their neighbors to algorithmically construct a set of multiple comparison systems, the successful construction of which can certify stability of the network. The rest of this article is organized as follows. Following some brief background in Section\,\ref{S:background} we describe the problem in Section\,\ref{S:problem}. We present the traditional approach to single comparison systems and an SOS-based direct method of computing the comparison systems in Section\,\ref{S:comparison}. In Section\,\ref{S:multiple}, we introduce the concept of multiple comparison systems, and propose a parallel and distributed algorithmic construction of the comparison systems in real-time. We demonstrate an application of this algorithm to a network of Van der Pol oscillators in Section\,\ref{S:results}, before concluding the article in Section\,\ref{S:concl}.

\section{Preliminaries}
\label{S:background}


Let us consider the dynamical system 
\begin{align}\label{E:f}
&\dot{x}\left(t\right) = f\left(x\left(t\right)\right),\quad t\geq 0,~x\in\mathbb{R}^n, ~f\left(0\right)=0\,,
\end{align}
with an equilibrium at the origin\footnote{Note that by shifting  the state variables any equilibrium point of interest can be moved to the origin.}, and $f:\mathbb{R}^n\rightarrow \mathbb{R}^n$ is locally Lipschitz. 
Let us use $\left|\,\cdot\,\right|$ to denote both the Euclidean norm (for a vector) and the absolute value (for a scalar). 
\begin{definition}\label{D:stability}
The equilibrium point at the origin is said to be asymptotically stable in a domain $\mathcal{D}\!\subseteq\!\mathbb{R}^n,\,0\!\in\!\mathcal{D},$ if $\lim_{t\rightarrow\infty}\left\vert x(t)\right\vert \!=\!0\,$ for every $\left\vert x(0)\!\right\vert \ \!\!\!\!\in\!\! \mathcal{D}$,
and it is exponentially stable if there exists $b,\,c \!>\! 0$ such that $\left\vert x(t)\right\vert \!<\!ce^{-bt}\!\left\vert  x(0)\right\vert \,\,\forall t\!\geq\! 0\,$, for every $\left\vert x(0)\!\right\vert \ \!\!\!\!\in\!\! \mathcal{D}$.
\end{definition}

\begin{theorem}\label{T:Lyap}
(Lyapunov, \cite{Lyapunov:1892}, \cite{Khalil:1996}, Thm. 4.1) If there exists a domain $\mathcal{D}\!\!\subseteq\!\!\mathbb{R}^n$, $0\!\in\!\!\mathcal{D}$, and a continuously differentiable positive definite function {$\tilde{V}\!\!:\!\mathcal{D}\!\rightarrow\! \mathbb{R}_{\geq 0}$}, i.e. the `Lyapunov function' (LF), then the equilibrium point of \eqref{E:f} at the origin is asymptotically stable if $\nabla{\tilde{V}}^T\!\!f(x)$ is negative definite in $\mathcal{D}$, and is exponentially stable if $\nabla{\tilde{V}}^T\!\!f(x)\leq\!-\alpha\, \tilde{V}~\forall x\!\in\!\mathcal{D}$, for some $\alpha>0$.
\end{theorem}
Here $\dot{\tilde{V}}(x)\!=\!\nabla{\tilde{V}}^T \!\!\cdot\! f(x)$. When there exists such a function $\tilde{V}\left(x\right)$, the region of attraction (ROA) of the stable equilibrium point at origin can be (conservatively) estimated as  \cite{Vicino:1985}
\begin{subequations}\label{E:ROA}
\begin{align}
&~~\mathcal{R}:=\left\lbrace x\in\mathcal{D}\left| {V}(x)\leq 1\right.\right\rbrace\,,~\text{with}~{V}(x)= {\tilde{V}(x)}/{\gamma^{max}},\\
&\text{where}~\gamma^{max}:=\max\left\lbrace \gamma\, \left\vert\, \left\lbrace x\in\mathbb{R}^n\left| \tilde{V}(x)\leq\gamma\right.\right\rbrace \subseteq \mathcal{D}\right.\right\rbrace\,,
\end{align}
\end{subequations}
i.e. the boundary of the ROA is estimated by the unit level-set of a suitably scaled LF ${V}(x)$.
Relatively recent studies have explored how sum-of-squares (SOS) based methods can be utilized to find LFs by restricting the search space to SOS polynomials \cite{Wloszek:2003,Parrilo:2000,Tan:2006,Anghel:2013}. 
Let us denote by $\mathbb{R}\left[x\right]$ the ring of all polynomials in $x\in\mathbb{R}^n$. 
\begin{definition}
A multivariate polynomial $p\in\mathbb{R}\left[x\right],~x\in\mathbb{R}^n$, is a sum-of-squares (SOS) if there exist some polynomial functions $h_i(x), i = 1\ldots s$ such that 
$p(x) = \sum_{i=1}^s h_i^2(x)$.
We denote the ring of all SOS polynomials in $x\in\mathbb{R}^n$ by $\Sigma[x]$.
\end{definition} 
Checking if $p\!\in\!\mathbb{R}[x]$ is an SOS is a semi-definite problem which can be solved with a MATLAB$^\text{\textregistered}$ toolbox SOSTOOLS \cite{sostools13,Antonis:2005a} along with a semidefinite programming solver such as SeDuMi \cite{Sturm:1999}. The SOS technique can be used to search for polynomial LFs by translating the conditions in Theorem\,\ref{T:Lyap} to equivalent SOS conditions \cite{sostools13,Wloszek:2003,Wloszek:2005,Antonis:2005,Antonis:2005a, Chesi:2010a }. An important result from algebraic geometry, called Putinar's Positivstellensatz theorem \cite{Putinar:1993,Lasserre:2009}, helps in translating the SOS conditions into SOS feasibility problems. 
The Putinar's Positivestellensatz theorem states  (see \cite{Lasserre:2009}, Ch. 2)
\begin{theorem}\label{T:Putinar}
Let $\mathcal{K}\!\!=\! \left\lbrace x\in\mathbb{R}^n\left\vert\, k_1(x) \geq 0\,, \dots , k_m(x)\geq 0\!\right.\right\rbrace$ be a compact set, where $k_j\!\in\!\mathbb{R}[x]$, $\forall j\in\left\lbrace 1,\dots,m\right\rbrace$. Suppose there exists a $\mu\!\in\! \left\lbrace \sigma_0 + {\sum}_{j=1}^m\sigma_j\,k_j \left\vert\, \sigma_0,\sigma_j \in \Sigma[x]\,,\forall j \right. \right\rbrace$ such that $\left\lbrace \left. x\in\mathbb{R}^n \right\vert\, \mu(x)\geq 0 \right\rbrace$ is compact. Then, if $p(x)\!>\!0~\forall x\!\in\!\!\mathcal{K}$, then $p \!\in\! \left\lbrace \sigma_0 \!\!+\!\! \sum_j\sigma_jk_j\!\!\left\vert\, \sigma_0,\sigma_j\!\!\in\!\Sigma[x],\forall j\!\right.\right\rbrace$.
\end{theorem}
\begin{remark}\label{R:Putinar}
Using Theorem\,\ref{T:Putinar}, we can translate the problem of checking that $p\!>\!0$ on $\mathcal{K}$ into an SOS feasibility problem where we seek the SOS polynomials $\sigma_0\,,\,\sigma_j\,\forall j$ such that $p\!-\!\sum_j\sigma_j k_j$ is {SOS.
Note that any} equality constraint $k_i(x)\!=\!0$ can be expressed as two inequalities $k_i(x)\!\geq 0$ and $k_i(x)\!\leq\! 0$. In many cases, especially for the $k_i\,\forall i$ used throughout this work, a $\mu$ satisfying the conditions in Theorem\,\ref{T:Putinar} is guaranteed to exist (see \cite{Lasserre:2009}), and need not be searched for.
\end{remark}

{In} \cite{Conti:1956,Brauer:1961} the authors proposed to view the LF as a dependent variable in a first-order auxiliary differential equation, often termed as the `comparison equation' (or, `comparison system'). 
{It was shown in \cite{Bellman:1962,Bailey:1966} that, under certain conditions,} the comparison equation can be effectively reduced to a set of linear differential equations. 
Noting that all the elements of the { matrix} $e^{At},~ t\geq 0$, where $A=\left[a_{ij}\right]\in\mathbb{R}^{m\times m}$, are non-negative if and only if $a_{ij}\geq 0, i\neq j$, {it was shown in \cite{Beckenbach:1961,Bellman:1962}:
\begin{lemma}\label{L:comparison}
Let $A\!\in\!\mathbb{R}^{m\times m}$ have non-negative off-diagonal elements, $v:[0,\infty)\!\rightarrow\!\mathbb{R}^m$ and $r:[0,\infty)\!\rightarrow\!\mathbb{R}^m$. If $v(0)\!=\!r(0)\,$, $\dot{v}(t)\!\leq\!Av(t)$ and $\dot{r}(t)\!=\!Ar(t)\,,$ then $v(t)\!\leq\! r(t)~\forall t\!\geq\! 0\,$.
\end{lemma}}
We henceforth refer to Lemma\,\ref{L:comparison} as the `comparison principle' and the {differential inequalities of the form $\dot{v}(t)\!\leq\!Av(t)$}
as a `comparison system' (CS). 

\section{Problem Description}\label{S:problem}

{Let us consider a network of $m$ (locally) asymptotically stable, polynomial\footnote{{We consider the cases when a non-polynomial dynamics can be recasted into an equivalent (exact) polynomial form, with additional variables and constraints \cite{Antonis:2005,Anghel:2013}. Otherwise, approximate polynomial forms (e.g. Taylor expansion) can be used \cite{Chesi:2009}.}} subsystems represented as follows,  
\begin{subequations}\label{E:fi}
\begin{align}
\forall i\!\in\!\lbrace 1,\dots,m\rbrace:~&\dot{x}_i = f_i(x_i) + g_i(x), ~ x_i\in\mathbb{R}^{n_i}, ~x\in\mathbb{R}^n\\
& g_i(x) \!=\! {\sum}_{j\neq i} g_{ij}(x_i, x_j)\,, \label{E:gij}\\
&f_i({0})\!=\!{0}\,,~g_{ij}(x_i,0)\!=\!0~\forall j\!\neq\! i\,,
\end{align}
\end{subequations}
where $x_i$ are the state variables that belong to the $i^{{th}}$ subsystem, $f_i\in\mathbb{R}[x_i]^{n_i}$ denotes the isolated subsystem dynamics, $g_i\in\mathbb{R}[x]^{n_i}$ represents the total neighbor interactions and $g_{ij}\in\mathbb{R}[x_i,x_j]^{n_i}$ quantifies how the $i^{th}$ subsystem affects the dynamics of the $j^{th}$ subsystem.} 
Note that we allow overlapping decompositions in which subsystems may share common state variable(s) \cite{Siljak:1978,Jocic:1977}, i.e. $n\!\leq\!{\sum}_{j=1}^m n_j\,$.
Finally, let
\begin{subequations}\label{E:Ni}
\begin{align}
\mathcal{N}_i &:= \left\lbrace i\right\rbrace\bigcup\left\lbrace j\left\vert ~\exists \, \left\lbrace x_i,x_j\right\rbrace, ~\text{s.t.}~g_{ij}\left(x_i,x_j\right)\neq 0 \right.\right\rbrace, \\
\text{and }\,\bar{x_{i}} &:= {\bigcup}_{j\in\mathcal{N}_i}\,\left\lbrace x_j\right\rbrace\,,
\end{align}\end{subequations}
denote {the set of the subsystems in the neighborhood of the $i^{th}$ subsystem (including itself)} and the state variables that belong to this neighborhood, respectively.

{The polynomial LFs, $V_i \in \mathbb{R}\left[x_i\right]~\forall i\,$, for the isolated subsystems, $\dot{x}_i = f_i(x_i)\,,\,\forall i\,$, are computed using an SOS-based \textit{expanding interior algorithm} \cite{Wloszek:2003,Anghel:2013} (alternatively, the methods in \cite{Tibken:2000, Chesi:2011} can be used), with the isolated ROAs
 \begin{align}
\mathcal{R}_i^0:= \left\lbrace x_i\in\mathbb{R}^{n_i}\left| V_i(x_i)\leq 1\right.\right\rbrace,~\forall i\!\in\!\lbrace 1,2,\dots,m\rbrace\,.
 \end{align}
The $V_i\,\forall i$ satisfy, for some ${\eta}_{i1},{\eta}_{i2},{\eta}_{i3}\!>\!0~\forall i$ and $\mathcal{D}_i\!\subset\!\mathcal{R}_i^0\,\forall i\,$,
\begin{subequations}\label{E:cond_VLF}
\begin{align}
\forall i:~\forall x_i\in\mathcal{D}_i\subset\mathcal{R}_i^0, ~&{\eta}_{i1}\left\vert x_i\right\vert^{d_i} \leq V_i(x_i) \leq {\eta}_{i2}\left\vert x_i\right\vert^{d_i} \,,\label{E:cond_VLF_1}\\
\text{and}~~&\nabla{V}_i^T\!\!f_i\, \leq -{\eta}_{i3}\left\vert x_i\right\vert^{d_i}\,, \label{E:cond_VLF_2}
\end{align}
\end{subequations}
w}here, $d_i$ is an even positive integer denoting the lowest degree among all monomials in $V_i(x_i)$. Further, the interaction terms, $g_{ij}\in\mathbb{R}[x_i,x_j]~\forall i,\forall j\in\mathcal{N}_i\backslash \lbrace i\rbrace$, satisfy the following,
\begin{align}\label{E:cond_inter}
\forall i\in&\left\lbrace 1,\dots,m\right\rbrace, \,\forall j\!\in\!\mathcal{N}_i\backslash\left\lbrace i\right\rbrace, ~\exists {\zeta}_{ij}>0\,~\text{such that,}\nonumber \\
&\forall x_i\!\in\!\mathcal{D}_i, \,\forall x_j\!\in\!\mathcal{D}_j, ~\left\vert \nabla{V}_i^T\!\!g_{ij}\right\vert \leq {\zeta}_{ij}\left\vert x_i\right\vert^{d_i-1}\left\vert x_j\right\vert.
\end{align}
{ We will address two stability problems in this paper. First, assume that a disturbance is applied to the link
between subsystems $i$ and $j$ (a fault in power systems). This means that $g_{ij}(x_i,x_j)$ in \eqref{E:gij} changes, or even becomes 0 if the link is cut, and the system moves away from its equilibrium point. 
After the fault is cleared, we reset the clock to 0, and consider the  evolution of system \eqref{E:fi} from the state $x(0) \neq 0$.}
Thus, any disturbance moves the system away from the equilibrium and results in positive level-sets $V_i(x_i(0))=v_i^0\in\left(0,1\right]$ for some or all of the subsystems. { A stability problem can be then formulated as checking if $\lim_{t\rightarrow +\infty}{V}_i(x_i(t))=0\,\forall i$ whenever $V_i(x_i(0))\!=\!v_i^0\,\forall i\,$,
where $x_i(t),~t>0$, are solutions of the coupled dynamics in \eqref{E:fi}. 
%
An attractive and scalable approach to solving this problem} is to construct a vector LF $V:\mathbb{R}^n\rightarrow\mathbb{R}_{\geq 0}^m$\,, defined as:
\begin{align}\label{E:vecLyap}
V(x) &:= \left[V_1(x_1)  ~~ V_2(x_2) ~~\dots ~~V_m(x_m)\right]\,^T , 
\end{align} 
and use a `comparison system' to certify that {$\lim_{t\rightarrow+\infty}V(x(t))=0\,.$}
Restricting our focus to the linear comparison principle (Lemma~\ref{L:comparison}), the aim is to seek an $A=[a_{ij}]\in\mathbb{R}^{m\times m}$ and a domain {$\mathcal{R}\!\subseteq\!\left\lbrace x\!\in\!\mathbb{R}^n\left\vert\,x_i\!\in\!\mathcal{R}_i^0~\forall i\right.\right\rbrace$, with $0\!\in\!\mathcal{R}$\,, such that
\begin{subequations}\label{E:comparison}
\begin{align}
\dot{V}(x)&\leq~ AV(x),~\forall x\in\mathcal{R}, \label{E:comparison_VAV}\\
\text{where,}\quad & \text{$A=[a_{ij}]$ is Hurwitz, }a_{ij}>0~\forall i\neq j\, , \\
\text{and}\quad & \text{$\mathcal{R}$ is invariant under the dynamics \eqref{E:f}.}
\end{align}
\end{subequations}
Henceforth, we refer to a comparison system of the form \eqref{E:comparison_VAV} as a \textit{`single comparison system'}, since one matrix $A=[a_{ij}]$ satisfies the differential inequalities in the full domain $\mathcal{R}$ which includes the origin.
%
When \eqref{E:comparison} holds\footnote{$A\!=\![a_{ij}]$ is called Hurwitz if its eigenvalues have negative real parts. $\mathcal{R}$ is called invariant if $x(0)\!\in\!\mathcal{R}\implies x(t)\!\in\!\mathcal{R}~\forall t\!>\!0$.}, any $x(0)\!\in\!\mathcal{R}$ would imply exponential convergence of $V(x(t))$ to the origin (from Lemma~\ref{L:comparison}), which, via \eqref{E:cond_VLF_1}, also implies exponential convergence of the states \cite{Siljak:1972}.
}

{ A second stability problem is to seek an optimal estimate of the region of attraction (ROA) of the stable equilibrium point by maximizing the domain $\mathcal{R}$ in \eqref{E:comparison}. While such optimization problems are difficult, and are not the main scope of this paper, we will describe in Section \ref{S:results} the results of  
estimating the ROA for a particular optimization direction.}

\section{Single Comparison System}
\label{S:comparison}
In this section, we first review the traditional approach towards stability analysis of interconnected dynamical systems using a {\textit{single comparison system (or, single CS)}, and then present an SOS-based approach that circumvents some of the issues with applicability of the traditional approach}. 

\subsection{Traditional Approach}\label{S:stab_old}
In \cite{Siljak:1972,Weissenberger:1973,Araki:1978,Jocic:1978}, and related works, authors laid out a formulation of the linear CS using certain conditions on the LFs and the neighbor interactions. It was observed that if there exists a set of LFs, $v_i:\mathbb{R}^{n_i}\rightarrow\mathbb{R}_{\geq 0}\,,~\forall\,i=1,2,\dots,m,$ satisfying the following conditions
{
\begin{subequations}\label{E:cond_Weissenberger}
\begin{align}
\forall (i,j)\!:\quad\tilde{\eta}_{i1}\left\vert x_i\right\vert  \leq v_i(x_i) \leq \tilde{\eta}_{i2}\left\vert x_i\right\vert\,,~\forall x_i\in\tilde{\mathcal{D}}_i \!\!\subset\!\mathcal{R}_i^0,  &\label{E:cond_VLF_old_1}\\
\left(\nabla{v}_i\right)^T\!\!f_i\, \leq -\tilde{\eta}_{i3}\left\vert x_i\right\vert\,,~\forall x_i\in\tilde{\mathcal{D}}_i \!\!\subset\!\mathcal{R}_i^0,  &\label{E:cond_VLF_old_2}\\
\text{and}~\left\vert \left(\nabla{v}_i\right)^T\!\!g_{ij}\right\vert  \leq \tilde{\zeta}_{ij}\left\vert x_j\right\vert ,~\forall x_i\!\in\!\tilde{\mathcal{D}}_i, \,\forall x_j\!\in\!\tilde{\mathcal{D}}_j, &\label{E:cond_inter_old}
\end{align}\end{subequations}
for some $\tilde{\eta}_{i1},\tilde{\eta}_{i2},\tilde{\eta}_{i3}\!>\!0~\forall i$ and $\tilde{\zeta}_{ij}\!\geq\!0~\forall (i,j)$, with $\tilde{\zeta}_{ij}\!=\!0~\forall j\!\notin\!\mathcal{N}_i$, then the corresponding vector LF $v\!=\![v_1~v_2\dots v_m]^T$ satisfies a CS on the domain $\tilde{\mathcal{D}}\!=\!\left\lbrace x\!\in\!\mathbb{R}^n\!\left|\, x_i\in\tilde{\mathcal{D}}_i\,\forall i\! \right.\right\rbrace$, with a comparison matrix $\tilde{A}\!=\![\tilde{a}_{ij}]$ given by
\begin{align}\label{E:cond_A}
\forall (i,j)\!:~\,\tilde{a}_{ii} \!=\! -{\tilde{\eta}_{i3}} /{\tilde{\eta}_{i2}}\,,~ \tilde{a}_{ij}\!=\!{\tilde{\zeta}_{ij}} /{\tilde{\eta}_{j1}}\,.
\end{align}
If $\tilde{A}$ is Hurwitz, then any invariant domain $\mathcal{R}\!\subseteq\!\tilde{\mathcal{D}}$ is an estimate of a region of exponential stability \cite{Weissenberger:1973,Jocic:1978}.} 

While the traditional approach provides very useful analytical insights into the construction of the comparison matrix $\tilde{A}=[\tilde{a}_{ij}]$, it unfortunately suffers from certain limitations, primarily due to the unavailability of suitable computational methods at that time. For example, the traditional approach requires finding the bounds in \eqref{E:cond_Weissenberger}, and also the LFs $v_i,\,\forall\,i,$ that satisfy those bounds. From the LFs $V_i\,\forall\,i$ satisfying \eqref{E:cond_VLF}, we can construct (non-polynomial) LFs \cite{Weissenberger:1973,Jocic:1978}, as:
\begin{subequations}\label{E:LF_Weiss}
\begin{align}
\!\!\!\!\!\!\forall (i,j):~v_i\!=\!\sqrt[d_i]{{V_i}}\text{ satisfies \eqref{E:cond_Weissenberger} with }&\!\!\!\!\\
\!\!\!\!\!\!\tilde{\eta}_{i1} \!=\! \sqrt[d_i]{{\eta}_{i1}} \,,~ \tilde{\eta}_{i2} \!=\! \sqrt[d_i]{{\eta}_{i2}} \, ,~\tilde{\eta}_{i3} \!=\! \frac{{\eta}_{i3}\tilde{\eta}_{i2}}{d_i\,\eta_{i2}} \,, ~\tilde{\zeta}_{ij} \!=\! \frac{{\zeta_{ij}}\tilde{\eta}_{i1}}{d_i\,\eta_{i1}}\,.\!\!\!\!
\end{align}
\end{subequations}
Thus the computation of each element of the comparison matrix $\tilde{A}$ in \eqref{E:cond_A} requires multiple optimization steps. It may also be noted that some of the bounds in \eqref{E:cond_Weissenberger}, while convenient for analytical insights, need not be optimal for computing a Hurwitz comparison matrix. For example, in \eqref{E:cond_inter_old}, $\left\vert\nabla v_i^T g_{ij}\right\vert$ is function of both $x_i$ and $x_j$ but is bounded by using only the norm on $x_j$.  


\subsection{SOS-Based Direct Computation}\label{S:stab_new}

{SOS-based techniques can be used to resolve some of the issues that arise with the traditional approach \cite{Kundu:2015CDC}}. The idea is to compute the \textit{single CS} in a decentralized way, using directly the LFs $V_i\in\mathbb{R}[x_i]$ \cite{Wloszek:2003,Anghel:2013}, which however do not satisfy the conditions in \eqref{E:cond_Weissenberger}. {For convenience, let us introduce, for all $i\in\lbrace 1,\dots,m\rbrace$, the following notations, 
\begin{subequations}\label{E:Di}\begin{align}
\!\!\!\!\!\!\forall\, 0\!\leq\!a_2\!<\!a_1\!\leq\!1\!:\,\mathcal{D}_i[a_1]\!:=\!\left\lbrace x\!\in\!\mathbb{R}^n\left\vert\, V_i(x_i)\!\leq\! a_1\!\!\right.\right\rbrace\!,\!\!&\!\!\!\\
\!\!\!\!\!\!\mathcal{D}_i^b[a_1]\!:=\!\left\lbrace x\!\in\!\mathbb{R}^n\left\vert\, V_i(x_i)\!=\! a_1\!\!\right.\right\rbrace\!,\!\!&\!\!\!\\
\!\!\!\!\!\!\text{and}~\mathcal{D}_i[a_1,a_2]\!:=\!\left\lbrace x\!\in\!\mathbb{R}^n\left\vert\, a_2\!<\! V_i(x_i)\!\leq\! a_1\!\!\right.\right\rbrace\!.\!\!&\!\!\!
\end{align}\end{subequations}
Given a set of $\gamma_i^0\!\in\!(0,1]\,\forall i\,$, we want to construct the single CS 
in a distributed way by calculating each row of the comparison matrix $A\!\in\!\mathbb{R}^{m\times m}\!$ (with non-negative off-diagonal entries) locally at each subsystem level, such that,
\begin{align}\label{E:Vaij}
\forall i\!\in\!\lbrace 1,\dots,m\rbrace:~\,\dot{V}_i\leq\! {\sum}_{j\in\mathcal{N}_i} a_{ij}V_j\, ~\text{on}~{\bigcap}_{j\in\mathcal{N}_i}\!\mathcal{D}_j[\gamma_j^0]\,,
\end{align}
}
\begin{proposition}\label{P:Gershgorin}
The domain ${\bigcap}_{i=1}^m\!\mathcal{D}_i[\gamma_i^0]$ is an estimate of the ROA of the interconnected system in \eqref{E:fi} if for each $i\in\lbrace 1,2,\dots,m\rbrace$, ${\sum}_{j\in\mathcal{N}_i}a_{ij}<0$ and ${\sum}_{j\in\mathcal{N}_i}a_{ij}\,\gamma_j^0\!<\!0\,$.
\end{proposition}
\begin{proof}
{Because of the non-negative off-diagonal entries and ${\sum}_{j\in\mathcal{N}_i}a_{ij}<0~\forall i\,$, the application of Gershgorin's Circle theorem \cite{Bell:1965,Gershgorin:1931} states that the comparison matrix $A\!=\![a_{ij}]$ is Hurwitz. 
Further, note that whenever $V_i(x_i(\tau))\!=\!\gamma_i^0$, for some $i\,$, and $V_k(x_k(\tau))\!\leq\!\gamma_k^0~\forall k\!\neq\!i\,$, for some $\tau\!\geq\! 0$, we have
$\left.\dot{V}_i\left(x_i\right)\right\vert_{t=\tau} \!<\!0\,.$
i.e. the (piecewise continuous) trajectories can never cross the boundaries defined by ${\bigcap}_{i=1}^m\!\mathcal{D}_i^b[\gamma_i^0]$. }
\hfill\hfill\qed
\end{proof}
\begin{remark}\label{R:invariance_Hurwitz}
Henceforth, we will loosely refer to the conditions of the form ${\sum}_{j\in\mathcal{N}_i}a_{ij}<0$ as the `Hurwitz conditions', while the conditions of the form ${\sum}_{j\in\mathcal{N}_i}a_{ij}\,\gamma_j^0\!<\!0\,$ will be referred to as the `invariance conditions'.
\end{remark}
{Proposition\,\ref{P:Gershgorin} helps us formulate \textit{local} (subsystem-level) SOS problems to find the single CS and check if the domain ${\bigcap}_{i=1}^m\!\mathcal{D}_i[\gamma_i^0]$ is an estimated ROA.} Note that in this formulation we use the polynomial LFs that do not satisfy the bounds in \eqref{E:cond_Weissenberger}. However, we prefer to directly use the polynomial LFs, instead of converting them to their non-polynomial counterparts as in \eqref{E:LF_Weiss}, for two reasons: 1) convenience of applying SOS methods, and 2) better stability certificates, as shown in the following result. 
{\begin{proposition}\label{L:equivalence}
If for some LFs $v_i\,,\,i\!\in\!\!\left\lbrace 1,\!\dots\!,m\right\rbrace$, there exists a comparison matrix $\tilde{A}\!=\![\tilde{a}_{ij}]$, with $\,\tilde{a}_{ii} \!+\! \sum_{j\neq i} \tilde{a}_{ij}\,\tilde{c}_{ij}\!<\!0~\forall i\!\in\!\lbrace 1,\dots,m\rbrace\,$, for some $\tilde{c}_{ij}\!>\!0\,\,\forall i\!\neq\! j$, then, for any LFs $V_i=\left(v_i\right)^{d}\,\,\forall\,i\,$, $d\!\geq\!1\,,$ the existence of a comparison matrix $A\!=\![a_{ij}]$ is guaranteed, with $\,{a}_{ii} \!+\! \sum_{j\neq i} {a}_{ij}\,\left(\tilde{c}_{ij}\right)^d\!<\!0~\forall i\,$.
\end{proposition}}
\begin{proof}
Note from Proposition\,\ref{P:Gershgorin}, that choosing $\tilde{c}_{ij}\!=\!1,$ $\gamma_j^0/\gamma_i^0,$ or $\max(\gamma_j^0/\gamma_i^0,\,1)$ we may retrieve the Hurwitz condition, the invariance condition, or simultaneously both, respectively.
{The proof follows directly after we show that $\forall i\!\in\!\lbrace 1,\dots,m\rbrace$,
\begin{align*}
\dot{V}_i &\leq d v_i^{d-1}\sum_{j=1}^m \tilde{a}_{ij}v_j = d\tilde{a}_{ii}V_i + d\sum_{j\neq i}\tilde{a}_{ij}\tilde{c}_{ij}V_i^{\frac{d-1}{d}}{\left(\tilde{c}_{ij}^{-d}V_j\right)^{1/d}}\notag\\
&\leq d\tilde{a}_{ii}V_i + {\sum}_{j\neq i}\tilde{a}_{ij}\tilde{c}_{ij}\left(\left(d\!-\!1\right)V_i + \tilde{c}_{ij}^{-d}V_j\right) = {\sum}_{j}a_{ij}V_j\,,
\end{align*}
by using Young's inequality\footnote{$a^{1/p} b ^{1/q} \leq a/p + b/q$ for $a,b >0, p > 1$ and $1/p+1/q = 1$.}\!\!, and choosing $a_{ii}\!=\!d\tilde{a}_{ii}\!+\!(d\!-\!1)\sum_{j\neq i}\tilde{a}_{ij}\tilde{c}_{ij}\,,\forall i\,,$ and $a_{ij}\!=\!\tilde{a}_{ij}\left(\tilde{c}_{ij}\right)^{1-d}~\forall i\!\neq\! j$\,.}\hfill\hfill\qed
\end{proof}
{Motivated by Propositions\,\ref{P:Gershgorin}-\ref{L:equivalence}\,, we propose an SOS-based direct computation of the 
comparison matrix $A\!=\![a_{ij}]$ in \eqref{E:Vaij}, in which each subsystem calculates the corresponding row entries of the matrix $A$, by solving the following SOS feasibility problem (using Theorem\,\ref{T:Putinar}):}
\begin{subequations}\label{E:sos_A}
\begin{align}
-\!\nabla V_i^T\!\!\left(f_i\!+\! g_i\right) + \!\!\!{\sum}_{j\in\mathcal{N}_i} \!\!\left(a_{ij}V_j \!- \sigma_{ij}\!\left(\gamma_j^0\!-\!V_j \right)\!\right) \!\in \Sigma[\bar{x}_i], & \label{E:Vaij_SOS}\\
 -{\sum}_{j\in\mathcal{N}_i}\,a_{ij}\in\Sigma[0]\,, & \label{E:TwoCond_Hurwitz}\\
 ~\text{and}~ \,-{\sum}_{j\!\in\!\mathcal{N}_i}\,a_{ij}\,\gamma_j^0\in\Sigma[0]\,, & \label{E:TwoCond_Invariance}\\
\text{where,}~\sigma_{ij}\!\in\!\Sigma[\bar{x}_i]\,\forall \!j\!\!\in\!\!\mathcal{N}_i\,,~a_{ii}\!\in\!\mathbb{R}[0]\,, &\\
\text{and}~\,a_{ij}\!\in\!\Sigma[0]\,\forall\! j\!\!\in\!\!\mathcal{N}_i\backslash\!\lbrace i\rbrace\,.&
\end{align}\end{subequations}
Here $\mathbb{R}[0]$ denotes scalars, $\Sigma[0]$ denotes non-negative scalars and $\bar{x}_i$ were defined in \eqref{E:Ni}. If \eqref{E:sos_A} is feasible for each $i\!\in\!\lbrace 1,\dots,m\rbrace$, then the origin is exponential stable and the domain $\bigcap_{i=1}^m\mathcal{D}_i[\gamma_i^0]$ is an estimated ROA.

\begin{remark}
Alternative, and possibly less conservative, formulations of \eqref{E:sos_A} are possible by replacing some of the constraints by an equivalent objective, e.g. replacing \eqref{E:TwoCond_Hurwitz} by an objective function $\min\,\sum_{j\in\mathcal{N}_i}a_{ij}\,$. But those will require centralized computations, e.g. finding eigenvalues of $A$, and are therefore omitted from further discussion.
\end{remark}

\subsection{Limitations}
{The direct computational approach using the \textit{single CS}, proposed in \cite{Kundu:2015CDC}, has certain limitations, both conceptual as well as computational, which may affect its applicability.} First of all, finding a single set of scalars $a_{ij}\,\forall i,\,\forall j\!\in\!\mathcal{N}_i\,,$ satisfying {the inequalities in a large domain} $\bigcap_{i=1}^m\mathcal{D}_i[\gamma_i^0]$ could be difficult. Note that the success of the CS approach relies on the values of the following:
\begin{align}\label{E:self_decay}
\alpha_i(\gamma_i^0)\!=\!\max\left\lbrace \alpha\!\geq\! 0\left\vert \,\nabla V_i^T\!\!f_i(x_i)\!\leq\!\!-\alpha V_i~\text{on}~\mathcal{D}_i^b[\gamma_i^0]\right.\!\!\right\rbrace ,
\end{align}
which we refer to as the `self-decay rates' of the isolated subsystem LFs. The function $\alpha_i:(0,1]\rightarrow \mathbb{R}_{\geq 0}$, with $\alpha_i(1)\!=\!0\,$,\footnote{Under fairly generic assumptions~\cite{Wu:1988}, for each $i$, $f_i(x_i)\!=\!0$ for at least some $x_i$ on the \textit{true}  boundary of the isolated ROA, while  $\nabla{V}_i(x_i)\!=\!0$ for the computed LF $V_i$ for at least some critical points
$x_i\!\in\!\mathcal{D}_i^b[1]$, which is the \textit{estimated} boundary of the isolated ROA.} is generally non-monotonic (as illustrated in the example in Section\,\ref{S:results}). Consequently, if the neighboring subsystems are at large level-sets, then a subsystem may not find suitable row-entries of the comparison matrix that satisfy both the conditions \eqref{E:TwoCond_Hurwitz}-\eqref{E:TwoCond_Invariance} over all of its level-sets down to the origin. In such a case, a more generalized approach, involving \textit{multiple comparison systems (or, multiple CSs)}, is necessary. Secondly, the proposed direct approach requires solving subsystem-level SOS problems that involve all the state variables associated with the neighborhoods. Consequently, presence of a large neighborhood can severely affect the overall computational speed, and scalability, of the analysis. A pairwise-interactions based approach can further reduce the computational burden at the subsystems { by reducing the size of the SOS problem}. {Finally, it is not clear how to find the scalars $\gamma_i^0\,\forall i$ that define ${\bigcap}_{i=1}^m\!\mathcal{D}_i[\gamma_i^0]$, the domain of definition of the CS. It is logical that the set of values for $\gamma_i^0\,\forall i$ should be found adaptively, given a disturbance, so that the domain ${\bigcap}_{i=1}^m\!\mathcal{D}_i[\gamma_i^0]$ takes a shape that resembles the particular disturbance. In Section\,\ref{S:multiple} we propose a novel analysis framework that attempts to resolve the above-mentioned limitations.}

\section{Multiple Comparison Systems}
\label{S:multiple}
\begin{figure}[t]
      \centering
	\includegraphics[scale=0.18]{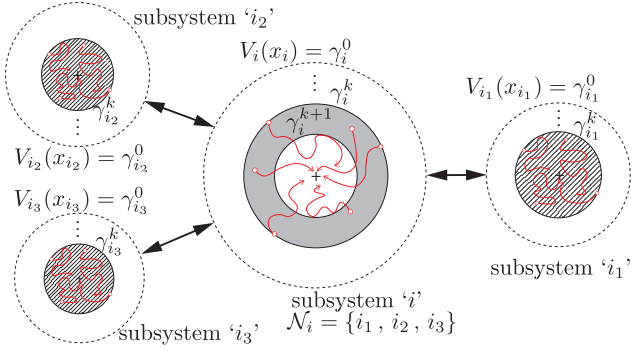}
      \caption{Distributed coordinated sequential stability certification .}
      \label{F:ring}
   \end{figure}   
   In this section, we formulate a generalized CSs approach in which we use a sequence of CSs to collectively certify stability under given disturbances. We also propose a framework to parallelize the subsystem-level SOS problems using the pairwise interactions. Fig.\,\ref{F:ring} illustrates the basic idea behind our proposed formulation. Given any invariant domain $\bigcap_{i=1}^m\mathcal{D}_i[\gamma_i^k]$, $k\!=\!0,1,2,\dots,$ the idea is to find the next invariant domain $\bigcap_{i=1}^m\mathcal{D}_i[\gamma_i^{k+1}]\,$, with $\mathcal{D}_i[\gamma_i^{k+1}]\!\subseteq\!\mathcal{D}_i[\gamma_i^k]~\forall i\,,$ such that any trajectory starting from $\bigcap_{i=1}^m\mathcal{D}_i[\gamma_i^k]$ converges \textit{exponentially} on $\bigcap_{i=1}^m\mathcal{D}_i[\gamma_i^{k+1}]$. {This is done in a distributed way in which each subsystem computes its next invariant level-set and communicates that value to its neighbors, until all the sequences of level-sets converge (to zero, for asymptotic stability). The idea is that,
\begin{lemma}\label{L:multiple1}
If the subsystem LFs of the interconnected system \eqref{E:fi} satisfy the \textit{`multiple CSs'} given by,
\begin{align*}
\forall(k,i): \dot{V}_i\!&\leq\!{\sum}_{j\in\mathcal{N}_i}a_{ij}^k(V_j\!-\!\gamma_j^{k+1})~\text{on}~{\bigcap}_{j\in\mathcal{N}_i}\mathcal{D}_j[\gamma_j^k,\gamma_j^{k+1}]
\end{align*}
with ${\sum}_{j\in\mathcal{N}_i} a_{ij}^k\!<\!0\,\forall (k,i)$ and ${\sum}_{j\in\mathcal{N}_i} a_{ij}^0(\gamma_j^0\!-\!\gamma_j^{1})\!<\!0\,\forall i$\,, then the system trajectories converge exponentially to $\bigcap_{i=1}^m\mathcal{D}_i[\gamma_i^*]$ where $\gamma_i^*\,\forall i$ is the limit of the monotonically decreasing sequence of non-negative scalars $\left\lbrace \gamma_i^k\right\rbrace,\,k\!\in\!\lbrace 0,1,2,\dots\rbrace$\,.
\end{lemma}
\begin{proof}
Note that the CSs can be written compactly as,
\begin{align*}
\forall k:~\dot{V}\!\leq\!A^k(V\!-\!\gamma^{k+1})~\text{on}~{\bigcap}_{i=1}^m\mathcal{D}_i[\gamma_i^k,\gamma_i^{k+1}]\,, 
\end{align*}
where $A^k\!=\![a_{ij}^k]$ and $\gamma^{k+1}\!=\![\gamma_1^{k+1}~\gamma_2^{k+1}\,\dots\,\gamma_m^{k+1}]^T$\,. The conditions ${\sum}_{j\in\mathcal{N}_i} a_{ij}^k\!<\!0\,\forall (k,i)$ and ${\sum}_{j\in\mathcal{N}_i} a_{ij}^0(\gamma_j^0\!-\!\gamma_j^{1})\!<\!0\,\forall i$ imply that $A^k\,\forall k$ is Hurwitz and ${\bigcap}_{i=1}^m\mathcal{D}_i[\gamma_i^0]$ is invariant. Hence, for each $k$, the system trajectories starting inside $\bigcap_{i=1}^m\mathcal{D}_i[\gamma_i^k]$ converge exponentially to $\bigcap_{i=1}^m\mathcal{D}_i[\gamma_i^{k+1}]$, while always staying within $\bigcap_{i=1}^m\mathcal{D}_i[\gamma_i^0]$. 
{
Note that now the new state variables in the comparison system are $(V_i - \gamma_j^{k+1})$. As the subsystems cross the level sets $\gamma_i^{k+1}$, the comparison system changes.  Indeed,
let's say, the subsystems cross into $\gamma_i^{k+1}$ in the order  $\lbrace 1,2,3,\dots\rbrace$. After subsystem 1 crosses, the new comparison system is 
 \begin{align*}
\dot{\tilde{V}}_{2:m}&\leq A^k_{2:m}\,\tilde{V}_{2:m}\\
\text{where,}~\tilde{V}_{2:m}&=[(V_2 - \gamma_2^{k+1}, \dots (V_m - \gamma_m^{k+1})]^T,\\
A^k_{2:m} &= \left[ \begin{array}{ccc}a^k_{22} & \dots & a^k_{2m}\\
\vdots & \ddots & \vdots\\a^k_{m2} &\dots & a^k_{mm}\end{array}\right]\,.
\end{align*}
 Not that each matrix in the sequence $A^k_{2:m}, \ldots, A^k_{m:m} $ remains Hurwitz. Moreover, regardless of the order in which
 the subsystems cross the level sets $\gamma_i^{k+1}$, the sequence of $A$ matrices remains Hurwitz, thus proving 
 finite time convergence to the new level sets.
}\hfill\hfill\qed
\end{proof}
\begin{corollary}\label{C:exponential}
$\gamma_i^*\!=\!0\,\forall i$ implies exponential stability.
\end{corollary}
Note, however, that such a formulation is difficult to implement in a distributed algorithm, since the computation of the $i^{th}$ row of the comparison matrices at the iteration $k$ requires that subsystem-$i$ has knowledge of the $\gamma_j^{k+1}\forall j\!\in\!\mathcal{N}_i\backslash\lbrace i\rbrace$ of its neighbors. A possible approach could be, for each iteration-$k$, compute the $k$-th CS iteratively, i.e. using iterations within iterations. But in this work, we restrict ourselves to a simpler formulation, by seeking only diagonal comparison matrices.
\begin{lemma}\label{L:multiple}
If the subsystem LFs of the system in \eqref{E:fi} satisfy
\begin{align*}
\forall(k,i): \dot{V}_i\!&\leq\!a_{ii}^k(V_i\!-\!\gamma_i^{k+1})\text{ on }\mathcal{D}_i[\gamma_i^k,\gamma_i^{k+1}]{\bigcap}_{j\in\mathcal{N}_i\backslash\lbrace i\rbrace}\mathcal{D}_j[\gamma_j^k]
\end{align*}
with $a_{ii}^k\!<\!0\,\forall (k,i)$\,, then the system trajectories converge exponentially to $\bigcap_{i=1}^m\mathcal{D}_i[\gamma_i^*]$ where $\gamma_i^*\,\forall i$ is the limit of the monotonically decreasing sequence of non-negative scalars $\left\lbrace \gamma_i^k\right\rbrace,\,k\!\in\!\lbrace 0,1,2,\dots\rbrace$\,.
\end{lemma}
\begin{proof}
${\bigcap}_{i=1}^m\mathcal{D}_i[\gamma_i^k]$ are invariant $\forall k\,$. The rest is trivial.\qed
\end{proof}
\begin{proposition}\label{P:convergence}
Exponential stability, i.e. $\gamma_i^*\!=\!0\,\forall i$\,, is guaranteed via the multiple (diagonal) CSs approach, if and only if 
$\nabla V_i^T(f_i\!+\!g_i)\!<\!0\text{ on }\mathcal{D}_i^b[\gamma_i^k]\bigcap_{j\in\mathcal{N}_i\backslash\lbrace i\rbrace}\mathcal{D}_j[\gamma_j^k]~\forall (k,i)\,.$
\end{proposition}
\begin{proof}
Exponential stability requires $\gamma_i^{k+1}\!<\!\gamma_i^k\,\forall (k,i)$ which immediately proves \textit{necessity}. The \textit{sufficiency} follows from the continuity of the polynomial functions.\hfill\hfill\qed
\end{proof}
\begin{corollary}\label{C:convergence}
Exponential stability is guaranteed if $\alpha(\gamma_i^k)\gamma_i^k\!>\!\max_{x\in\mathcal{D}_i^b[\gamma_i^k]\bigcap_{j\in\mathcal{N}_i\backslash\lbrace i\rbrace}\mathcal{D}_j[\gamma_j^k]}\nabla V_i^T\!g_i\,~\forall (k,i)\,$.
\end{corollary}
The computation of the \textit{multiple CSs} in Lemma\,\ref{L:multiple} involves two phases. In \textit{Phase 1}, we search for the level-sets $\gamma_i^0\!\in\![0,1)\,\forall i$ such that the system trajectories starting from some initial level-sets $V_i(x_i(0))\!=\!v_i^0\!\in\![0,\gamma_i^0]\,\forall i,$ will always stay within $\bigcap_{i=1}^m\mathcal{D}_i[\gamma_i^0]$ which we term as an \textit{`invariant envelope'} of $\bigcap_{i=1}^m\mathcal{D}_i[v_i^0]$. In \textit{Phase 2}, we compute $a_{ii}^k$ and $\gamma_i^{k+1}\forall(k,i)$\,.}

\subsection{Distributed Construction}\label{S:distributed}

\subsubsection{Phase 1: Find the Invariant Envelope $\bigcap_i\mathcal{D}_i[\gamma_i^0]$}\label{S:phase1}

{We search for the smallest $\gamma_i^0\!\in\![v_i^0,1)\,\forall i$ that satisfy,  
\begin{align}\label{E:phase1}
\forall i\!:\quad&\dot{V}_i\leq \!0\,~\text{on}~\mathcal{D}_i^b[\gamma_i^0]{\bigcap}_{j\in\mathcal{N}_i\backslash\lbrace i\rbrace}\mathcal{D}_j[\gamma_j^0],
\end{align}
This} requires \textit{knowledge} of the neighbors' \textit{expanded} level-sets, and hence can only be solved via an iterative process which aims to achieve an agreement between the neighboring subsystems on their individual \textit{expanded} level-sets. {Setting $\hat{\gamma}_i^{0}\!=\!v_i^0\,\forall i$ we compute the monotonically increasing sequences $\lbrace \hat{\gamma}_i^{l}\rbrace\,\forall i\,,\,l\!\in\!\lbrace 0,1,2,\dots\rbrace$ satisfying the following, 
\begin{subequations}\begin{align}
\forall(l,i):\quad &\dot{V}_i\leq \!0\,~\text{on}~\,\mathcal{D}_i^b[\hat{\gamma}_i^{l+1}]{\bigcap}_{j\in\mathcal{N}_i\backslash\lbrace i\rbrace} \!\!\mathcal{D}_j[\hat{\gamma}_j^{l}], \label{E:phase1_iterative}\\
\text{(SOS)}:\quad &\!\!\!\!\!\!\left\lbrace\begin{array}{l}
-\!\nabla V_i^T\!\left(f_i\!+\!g_i\right)\!-\! \sigma_{ii}(\hat{\gamma}_i^{l+1}\!\!-\!V_i )  \\
	\quad\qquad- {\sum}_{j\in\mathcal{N}_i\backslash\lbrace i\rbrace} \sigma_{ij}(\hat{\gamma}_j^{l}\!\!\!-\!V_j ) \!\in \Sigma[\bar{x}_i]\,,  \\	
\sigma_{ii}\in\mathbb{R}[\bar{x}_i],~\sigma_{ij}\in\Sigma[\bar{x}_i]\,\forall \!j\!\neq\!i\,,\end{array}\!\!\!\!\right.\!\!\!\!
\end{align}\end{subequations}
which is accomplished by finding the smallest $\hat{\gamma}_i^{l+1}$, using an \textit{incremental-search} approach\footnote{$\hat{\gamma}_i^{l+1}$ is increased in small steps until the SOS problem is feasible.} to handle the bilinear term in $\sigma_{ii}$ and $\hat{\gamma}_i^{l+1}$. 
If $\lbrace \hat{\gamma}_i^{l}\rbrace\,\forall i$ converge at some $l\!=\!L\,$, we assign $\gamma_i^0\!\gets\! \hat{\gamma}_i^{L}\,\forall i$
and stop. }

\subsubsection{Phase 2: Find the Diagonal Comparison Matrices}\label{S:phase2}

{With the \textit{invariant envelope} already found, we can compute the \textit{multiple CSs} in Lemma\,\ref{L:multiple}. At each iteration $k$, each subsystem-$i$ computes $a_{ii}^k$ and the smallest $\gamma_i^{k+1}\,$, using a \textit{bisection-search} on $\gamma_i^{k+1}$ over $[0,\,\gamma_i^k]$\,, which satisfy:
\begin{subequations}\label{E:phase2_step2}
\begin{align}
\!\!\!\!\!\!\!\!\forall(k,i):~& \dot{V}_i\!\leq\!a_{ii}^k(V_i\!-\!\gamma_i^{k+1})\text{ with }a_{ii}^k\!<\!0\,,\!\!\!\!\\
&\text{everywhere on }\mathcal{D}_i[\gamma_i^k\!,\gamma_i^{k+1}]{\bigcap}_{j\in\mathcal{N}_i\backslash\lbrace i\rbrace}\mathcal{D}_j[\gamma_j^k]\,. \!\!\!\!\notag\\
\!\!\!\!\!\!\!\!\text{(SOS)}:~&\!\!\!\!\!\!\left\lbrace\!\!\begin{array}{l} 
\!\!- \!\! \nabla V_i^T\!\!\left(f_i\!+\!g_i\right)\!+\!(a_{ii}^k-\underline{\sigma}_{ii})(V_i\!-\!\gamma_i^{k+1}) \\
 \qquad\qquad\qquad + {\sum}_{j\in\mathcal{N}_i} \sigma_{ij}(V_j \!-\! \gamma_j^{k})\in \Sigma[\bar{x}_i]\,,\\
\!-\!{a}_{ii}^{k}\!\in\!\Sigma[0],\,\underline{\sigma}_{ii}\in\Sigma[\bar{x}_i]\,,\,\sigma_{ij}\in\Sigma[\bar{x}_i]\,\,\forall j\!\in\!\mathcal{N}_i\,.\end{array}\right.\!\!\!\!\!\!\!\!
\end{align}\end{subequations}
We continue until the sequences $\lbrace \gamma_i^k\rbrace\,\forall i$ converge. The exponential stability is guaranteed if $\gamma_i^{k+1}\!\!=0\,\forall i$\,.}

\subsection{Distributed Parallel Construction}\label{S:parallel}

The computational complexities in the SOS problems in Section\,\ref{S:distributed} are largely dominated by the dimension of the state-space of the associated neighborhood. To circumvent this limitation, we propose a parallel formulation of the SOS problems based on the pairwise interaction terms. Note that, 

{\begin{lemma}\label{L:multiple_parallel}
If the subsystem LFs of the system in \eqref{E:fi} satisfy 
\begin{align*}
\forall(k,i,j):~ \nabla V_i^T(w_{ij}^kf_i+g_{ij})\!&\leq\!a_{ii,j}^k(V_i\!-\!\gamma_i^{k+1})\,,~w_{ij}^k\!\geq\!0\,,
\end{align*}
on $\mathcal{D}_i[\gamma_i^k,\gamma_i^{k+1}]{\bigcap}\mathcal{D}_j[\gamma_j^k]$, with $\sum_{j\in\!\mathcal{N}_i\backslash\lbrace i\rbrace}w_{ij}^k\!<\!1$ and $a_{ii,j}^k\!<\!0\,\forall (k,i)$\,, then the system trajectories converge exponentially to $\bigcap_{i=1}^m\mathcal{D}_i[\gamma_i^*]$ where $\gamma_i^*\,\forall i$ is the limit of the monotonically decreasing sequence of non-negative scalars $\left\lbrace \gamma_i^k\right\rbrace,\,k\!\in\!\lbrace 0,1,2,\dots\rbrace$\,.
\end{lemma}
\begin{proof}
Note that $\nabla V_i^T\!(f_i\!+\!\sum_jg_{ij})\!\leq\!\sum_j\alpha_{ii,j}^k(V_i\!-\!\gamma_i^{k+1})\!+\!(1\!-\!\sum_jw_{ij}^k)\nabla V_i^T\!\!f_i\!<\!\sum_j\alpha_{ii,j}^k(V_i\!-\!\gamma_i^{k+1})$. The rest is trivial.\hfill\hfill\qed
\end{proof}
N}ext we present an alternative formulation of the algorithmic steps in Section\,\ref{S:distributed} by finding these \textit{`weights'}, $w_{ij}^k$, and then using these \textit{weights} to parallelize the SOS problems.

\subsubsection{Phase 1: Find the Invariant Envelope $\bigcap_i\mathcal{D}_i[\gamma_i^0]$}\label{S:phase1_parallel}

{We set $\hat{\gamma}_i^{0}\!=\!v_i^0\,\forall i$ and compute the sequences $\lbrace \hat{\gamma}_i^{l}\rbrace\,\forall i$ by finding, for each subsystem-$i$\,, the smallest $\hat{\gamma}_i^{l+1}$ such that ${\sum}_{j\in\! \mathcal{N}_i\!\backslash\!\lbrace i\rbrace} \hat{w}_{ij}^{l}\!\!<\!\!1\,,$ where the \textit{`weights'}, $\hat{w}_{ij}^{l},$ are defined as, 
\begin{subequations}\label{E:phase1_parallel}
\begin{align}
\forall (l,i,j):~& \hat{w}_{ij}^{l}\!=\! \min\!\left\lbrace \hat{w}\left\vert\, \begin{array}{c} \nabla{V}_i^T\!(\hat{w} f_i\!+\!g_{ij})\!\leq\!0\\
					\text{on}~\!\mathcal{D}_i^b[\hat{\gamma}_i^{l+1}]\bigcap\mathcal{D}_j[\hat{\gamma}_j^{l}]\end{array}\!\right.\!\right\rbrace\!, \\
\text{(SOS)}:~&\left\lbrace\begin{array}{l} \underset{\sigma_{ii},\,\sigma_{ij}}{\text{minimize}}~\hat{w}\,,\,\text{subject to:}\\
\!-\!\nabla{V}_i^T\!(\hat{w} f_i\!+\!g_{ij})\!-\!\sigma_{ii}(\hat{\gamma}_i^{l+1}\!\!\!-\!V_i) \\
			\qquad\qquad-\sigma_{ij}(\hat{\gamma}_j^{l}\!-\!V_j)\in\Sigma[\,x_i\,,\,x_j], \\
			\!\sigma_{ii}\!\in\!\mathbb{R}[x_i,x_j],\,\sigma_{ij}\!\in\!\Sigma[x_i,x_j].
\end{array}\right. 	
\end{align}\end{subequations}
This is done using an incremental-search approach on $\hat{\gamma}_i^{l+1}$. If $\lbrace \hat{\gamma}_i^{l}\rbrace\,\forall i$ converge at $l\!=\!L\,$, we assign $\gamma_i^0\!\gets\! \hat{\gamma}_i^{L}\,\forall i$ and stop.}

\subsubsection{Phase 2: Find the Diagonal Comparison Matrices}\label{S:phase2_parallel}

{\textit{Phase 2} of the process essentially remains same as the one in described in Section\,\ref{S:phase2} except that we need to additionally compute the \textit{weights} $w_{ij}^k$\,. For each subsystem-$i$, we perform a \textit{bisection-search} on the smallest $\gamma_i^{k+1}$ over $[0,\,\gamma_i^k]$\,, such that $\sum_{j\in\mathcal{N}_i}w_{ij}^k\!<\!1$\,, where the weights $w_{ij}^k$ are defined as,
\begin{subequations}\label{E:phase2_parallel}
\begin{align}
\!\!\!\!\!\!\forall (k,i,j&):~w_{ij}^k\!:=\!\\
&\qquad\min\!\left\lbrace\! {w}\left\vert\, \begin{array}{c} \nabla{V}_i^T\!({w} f_i\!+\!g_{ij})\!\leq\!a_{ii,j}^k(V_i\!-\!\gamma_i^{k+1})\\
\!\text{on $\mathcal{D}_i[\gamma_i^k,\gamma_i^{k+1}]{\bigcap}\mathcal{D}_j[\gamma_j^k]$}\,,~a_{ii,j}^k\!<\!0.\end{array}\!\!\right.\!\!\right\rbrace\!\!,\!\!\!\!\!\!\notag\\
\!\!\!\!\!\text{(SOS)}\!:&\!\left\lbrace\!\!\begin{array}{l} 
\underset{a_{ii,j}^k,\,\underline{\sigma}_{ii},\,\sigma_{ii},\,\sigma_{ij}}{\text{minimize}}~\hat{w}\,,\,\text{subject to:}\\
\!\!- \!\! \nabla V_i^T\!\!\left(wf_i\!+\!g_{ij}\right)\!+ \! (a_{ii,j}^{k}\!-\!\underline{\sigma}_{ii})(V_i \!-\! \gamma_i^{k+1})\\
\qquad\qquad\quad+ {\sum}_{p=i,j}{\sigma}_{ip}(V_p\!-\!\gamma_p^k ) \in \Sigma[x_i,x_j]\,,\\
\!-\!{a}_{ii,j}^{k}\!\in\!\Sigma[0]\text{ and }\underline{\sigma}_{ii},\,\sigma_{ii},\,\sigma_{ij}\!\in\!\Sigma[x_i,x_j]\,.\end{array}\right.\!\!\!\!\!\!
\end{align}\end{subequations}
At each $k$\,, we solve the above bisection search to find the smallest $\gamma_i^{k+1}\,\forall i$ satisfying $\sum_{j\in\mathcal{N}_i}w_{ij}^k\!<\!1$\,, until $\lbrace\gamma_i^k\rbrace\,\forall i$ converges. If $\gamma_i^{k+1}\!=\!0\,\forall i$, the exponential stability is guaranteed.}

\begin{figure*}[thpb]
\centering
\subfigure[ROAs for isolated subsystem 9]{
\includegraphics[scale=0.36]{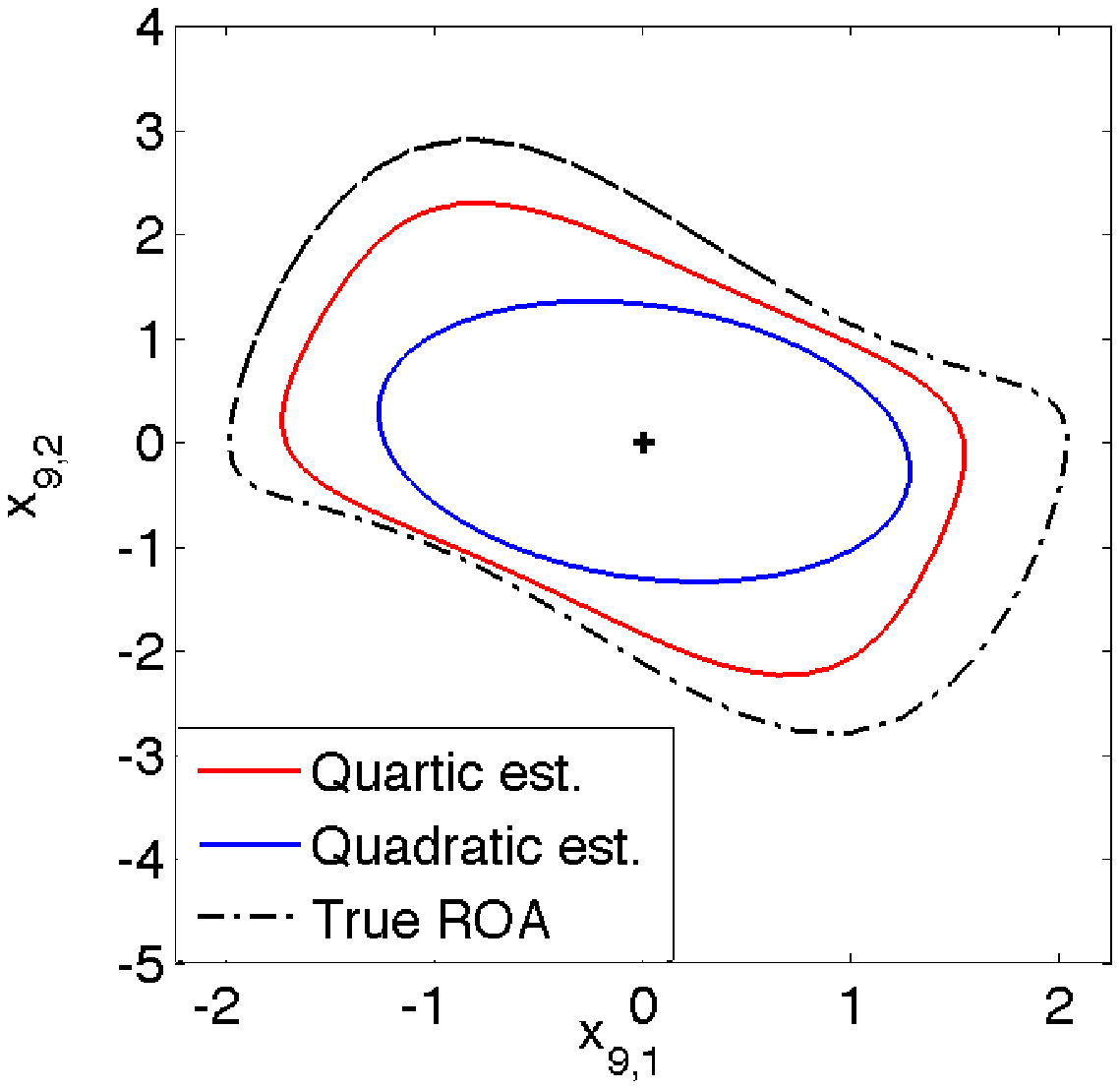}\label{F:ROAcompare}
}\hspace{0.001in}
\subfigure[`Self-decay' rates for quadratic LFs]{
\includegraphics[scale=0.36]{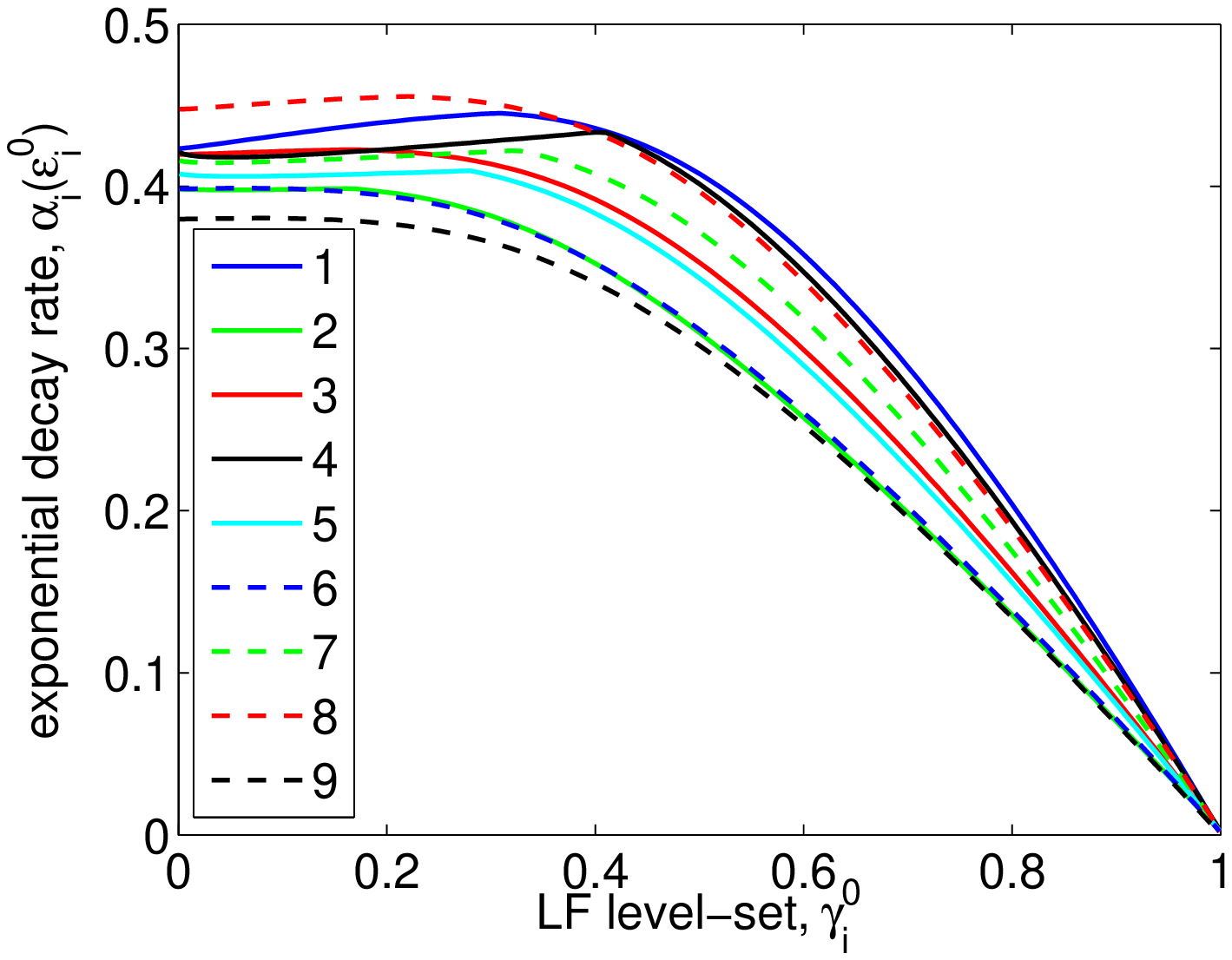}\label{F:decay_quad}
}\hspace{0.001in}
\subfigure[`Self-decay' rates for quartic LFs]{
\includegraphics[scale=0.36]{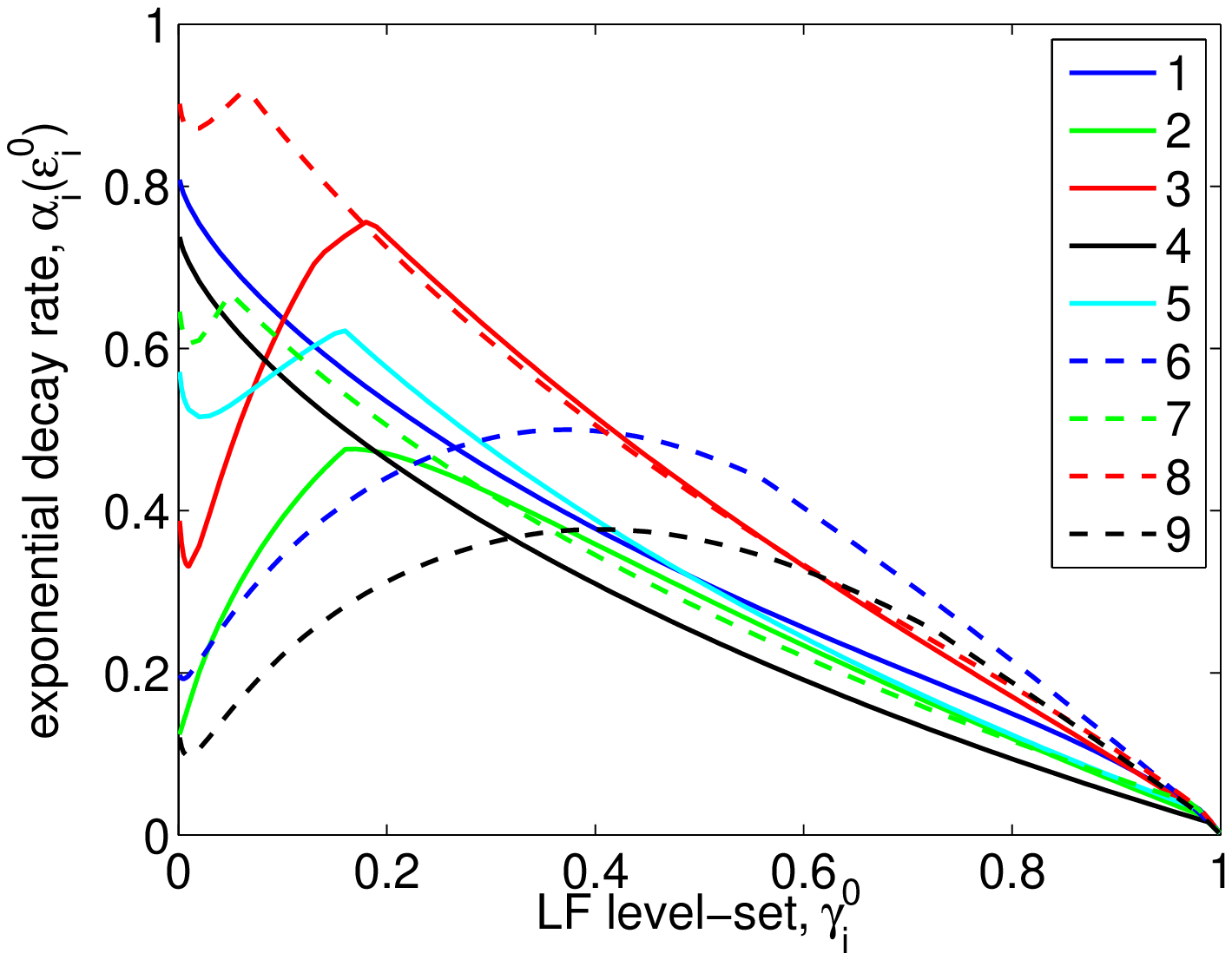}\label{F:decay_quart}
}\caption[]{Characteristics of subsystem LFs: (a) comparison of estimated ROAs, (b)-(c) exponential `self-decay' rates.}
\end{figure*}  
\section{Numerical Example}\label{S:results}
 We consider a network of nine modified Van der Pol `oscillators' \cite{van:1926}, with parameters of each oscillator chosen to make them individually stable. Each Van der Pol is treated as an individual subsystem, with the following interconnections,
\begin{align}
\begin{array}{lll}
\mathcal{N}_1:\left\lbrace 1, 2, 5, 9\right\rbrace & \mathcal{N}_2:\left\lbrace 2, 1, 3\right\rbrace & \mathcal{N}_3:\left\lbrace 3, 2, 8\right\rbrace\\ 
\mathcal{N}_4:\left\lbrace 4, 6, 7\right\rbrace & \mathcal{N}_5:\left\lbrace 5, 1, 6\right\rbrace & \mathcal{N}_6:\left\lbrace 6, 4, 5\right\rbrace\\
\mathcal{N}_7:\left\lbrace 7, 4, 8, 9\right\rbrace & \mathcal{N}_8:\left\lbrace 8, 3, 7\right\rbrace & \mathcal{N}_9:\left\lbrace 9, 1, 7\right\rbrace\,.
\end{array}
\end{align}
Each {subsystem $i\!\in\!\lbrace 1,2,\dots,9\rbrace$} has two state variables, $x_i=\left[\,x_{i,1}~\,x_{i,2}\,\right]^T$. After shifting the equilibrium point to the origin (see Appendix\,\ref{A:model} for details), the subsystem dynamics in the presence of the neighbor interactions are given by 
\begin{subequations}\label{E:VdP}
\begin{align}
&f_i(x_i)\!=\! \left[x_{i,2}\,,~\mu_i\,x_{i,2}(c_i^{(1)}\!\!-\!c_i^{(2)}x_{i,1}\!-\!x_{i,1}^2) \!-\! c_i^{(3)}x_{i,1}\right]^T \\
&				g_{ij}(x_i,x_j)\!=\! \left[0\,,~\beta_{ij}^{(1)}\,x_{j,2} + \beta_{ij}^{(2)}\,x_{j,2}\,x_{i,1}\right]\,.
\end{align}\end{subequations}
where, $c_i^{(1)}\!=\!1\!-\!\left(0.5\,c_i^{(2)}\right)^2$, $c_i^{(3)}\!=\!1\!-\!{\sum}_{j\in\mathcal{N}_i\lbrace i\rbrace}({0.5\,\beta_{ij}^{(2)}c_i^{(2)}}\!-\!\beta_{ij}^{(1)})$, $\mu_i\,,\,\beta_{ij}^{(1)}$ and $\beta_{ij}^{(2)}$ are chosen randomly and $c_i^{(2)}$ are related to the equilibrium point before shifting. 
Polynomial LFs for the isolated (no interaction) subsystems are computed using the \textit{expanding interior algorithm} (Section\,\ref{S:problem}). {Fig.\,\ref{F:ROAcompare} shows that a quartic LF estimates the `true' ROA of the isolated subsystems (obtained via time-reversed simulation) better than a quadratic LF. However, a better estimate of the isolated ROAs does not necessarily translate into better stability certificates for the interconnected system, as illustrated later.} 
The `self-decay rates', from \eqref{E:self_decay}, are plotted in Figs.\,\ref{F:decay_quad}-\ref{F:decay_quart} for a range of level-sets from $0$ to $1$. For each subsystem, as $\gamma_i^0$ approaches 1, $\alpha_i(\gamma_i^0)$ approaches 0. Thus it is impossible to obtain a Hurwitz comparison matrix when the initial conditions lie close to the boundary of the estimated ROAs. Moreover, note that the evolution of the self-decay rates is generally non-monotonic. { Thus an attempt to find a single CS valid all the way to the origin is generally difficult, since the row-sum values of the single comparison matrix will be limited by the lowest self-decay rate. In such a case, a multiple CS approach, however, can still guarantee exponential convergence to some level-sets close to the origin. 

We compare the traditional and the direct approaches of computing a CS (using the quadratic LFs), in Fig.\,\ref{F:CompMat}. Choosing $\gamma_1^0\!=\!\gamma_2^0\!=\!\dots\!=\!\gamma_9^0\,$, and varying their values, we compute the comparison matrices in \eqref{E:Vaij} using SOS-based direct approach in \eqref{E:sos_A}, by replacing the constraint \eqref{E:TwoCond_Hurwitz} by an objective of minimizing $\sum_j a_{ij}~\forall i\,$. Also, we find the comparison matrices using the traditional approach, from \eqref{E:LF_Weiss} and \eqref{E:cond_A}. 
\begin{figure}[thpb]
      \centering
	\includegraphics[scale=0.36]{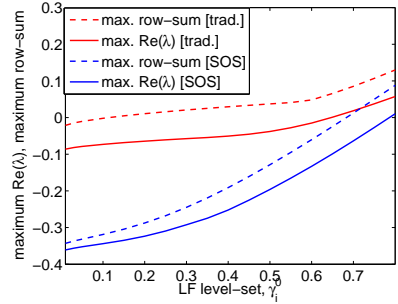}
	 \caption{{Stability properties of the single comparison matrices obtained via traditional and direct methods (using quadratic LFs})}
      \label{F:CompMat}
   \end{figure}   
The maximum of the real parts of the eigenvalues (denoted by $\text{Re}(\lambda)$) and the maximum row-sum of the comparison matrices are plotted, for varying level-sets. Clearly, the SOS-based direct method yields improved stability certificates. {Further note that the maximal (uniform) level-set for which the maximum row-sum value is negative gives \textit{an estimate} of the ROA of the full interconnected system. Thus $\bigcap_{i=1}^9\mathcal{D}[0.6831]$ is \textit{an estimate} of the ROA. Next we compare the performances of the single CS approach and the multiple CSs approach (with and without the parallel computation), in Fig.\,\ref{F:CompAlg}. For each subsystem-i, we plot the maximal (uniform) level-set for which either the row-sum is negative (single CS), or a strict convergence is achieved at iteration-0, i.e. $\gamma_i^1\!<\!\gamma_i^0$\, (multiple CS). When not considering the parallel implementation, the multiple CS approach outperforms the single CS in estimating the invariance (since we focus only at iteration-0). However, the parallel implementation, while achieving computational tractability for larger systems, yields more conservative certificates.

\begin{figure}[thpb]
      \centering
	\includegraphics[scale=0.4]{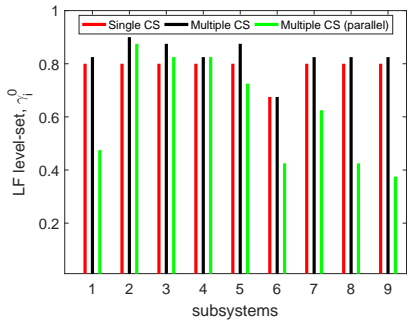}
      \caption{{Comparison of the single CS (`red') and multiple CSs approaches via the distributed construction ('black') and the distributed parallel construction ('green').}}
      \label{F:CompAlg}
   \end{figure}      
   } 

{Next we use an example to illustrate a couple of key observations. 
Fig.\,\ref{F:example_64} shows the stability analysis results on an arbitrarily generated initial condition (or disturbance) using both quadratic and quartic LFs, and multiple CSs. Note in Fig.\,\ref{F:LF_CS_quad_64} that the initial level-set lie outside the estimated ROA obtained in Fig.\,\ref{F:CompMat}. By allowing the analysis to be dependent on the particular initial condition, we are able to find a \textit{suitably} shaped stability region. 
Further note that, while the quadratic LFs-based multiple CSs analysis certifies exponential stability (in Fig.\,\ref{F:LF_CS_quad_64}), the quartic LFs can only guarantee exponential convergence to a domain $\bigcap_{i=1}^9\mathcal{D}_i[\gamma_i^*]$ very close to the origin, with $\gamma_2^*\!=\!0.023$ and $\gamma_3^*\!=\!0.016$ (in Fig.\,\ref{F:LF_CS_quart_64}). In fact, this domain of convergence is characteristic of the system and the (quartic) LFs used, and is independent of the initial condition. Referring to Fig.\,\ref{F:decay_quart}\,, the low self-decay rates of the quartic LFs for subsystems 2 and 3 explain the convergence away from the origin (Corollary\,\ref{C:convergence}). Thus, while the quartic LFs may yield better estimates of the isolated ROAs (Fig.\,\ref{F:ROAcompare}) compared to the quadratic LFs, and hence are able to analyze a larger number of initial conditions, they fail to certify exponential stability. This suggests that it may be useful to switch from quartic LFs to the quadratic LFs as the $\gamma_i^k\,\forall (k,i)$ decrease. We also noted that, in the quartic LF-based analysis subsystems-6 and 7 underwent an expansion (Phase 1) from $(v_6^0,v_7^0)\!=\!(0.015,0.002)$ to $(\gamma_6^0,\gamma_7^0)\!=\!(0.151,0.013)$, while none of the subsystems underwent expansion using quadratic LF-based analysis. 
}  } 

\begin{figure*}[thpb]
\centering
\subfigure[System states under disturbance]{
\includegraphics[scale=0.36]{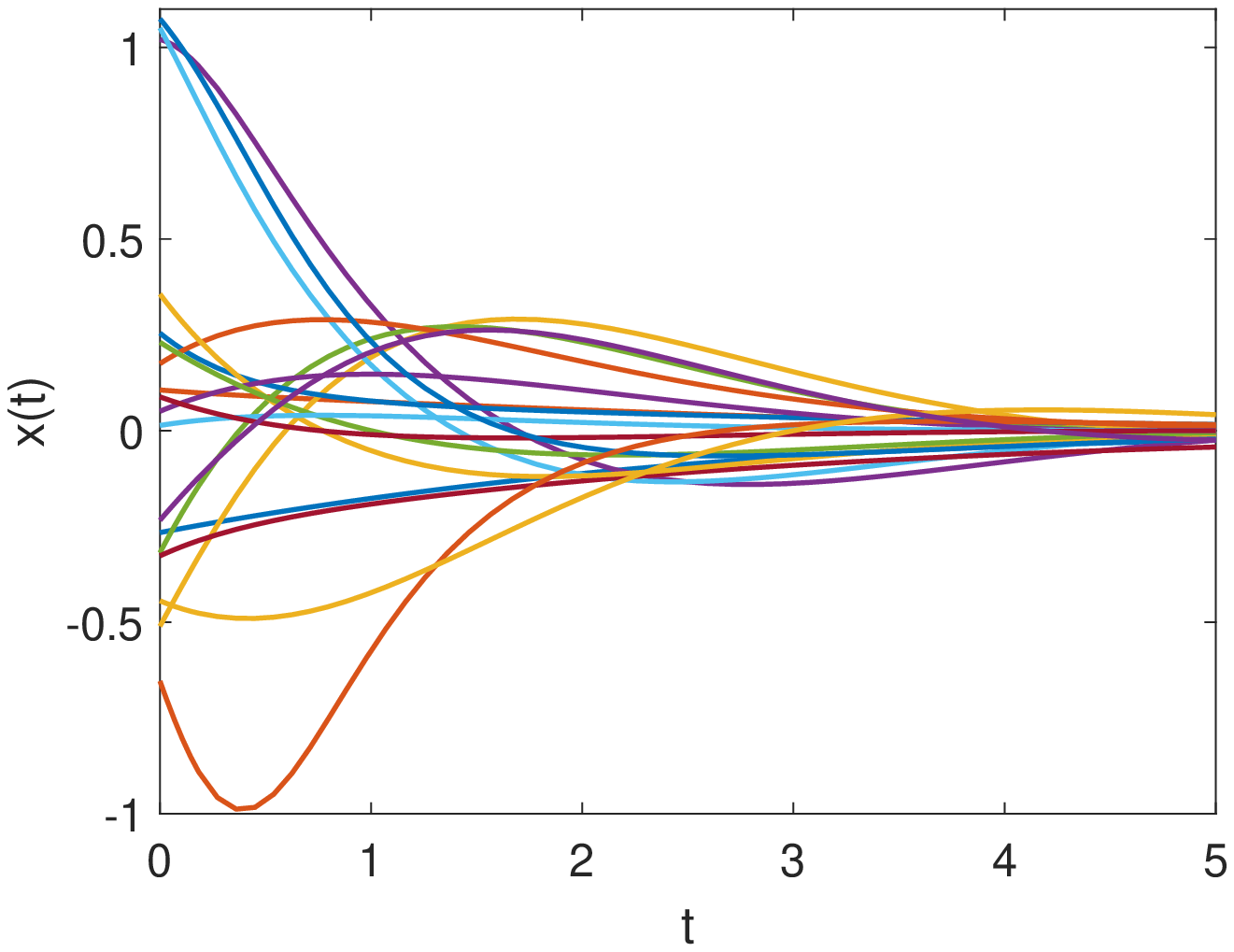}\label{F:states_64}
}\hspace{0.001in}
\subfigure[Quadratic LFs analysis]{
\includegraphics[scale=0.36]{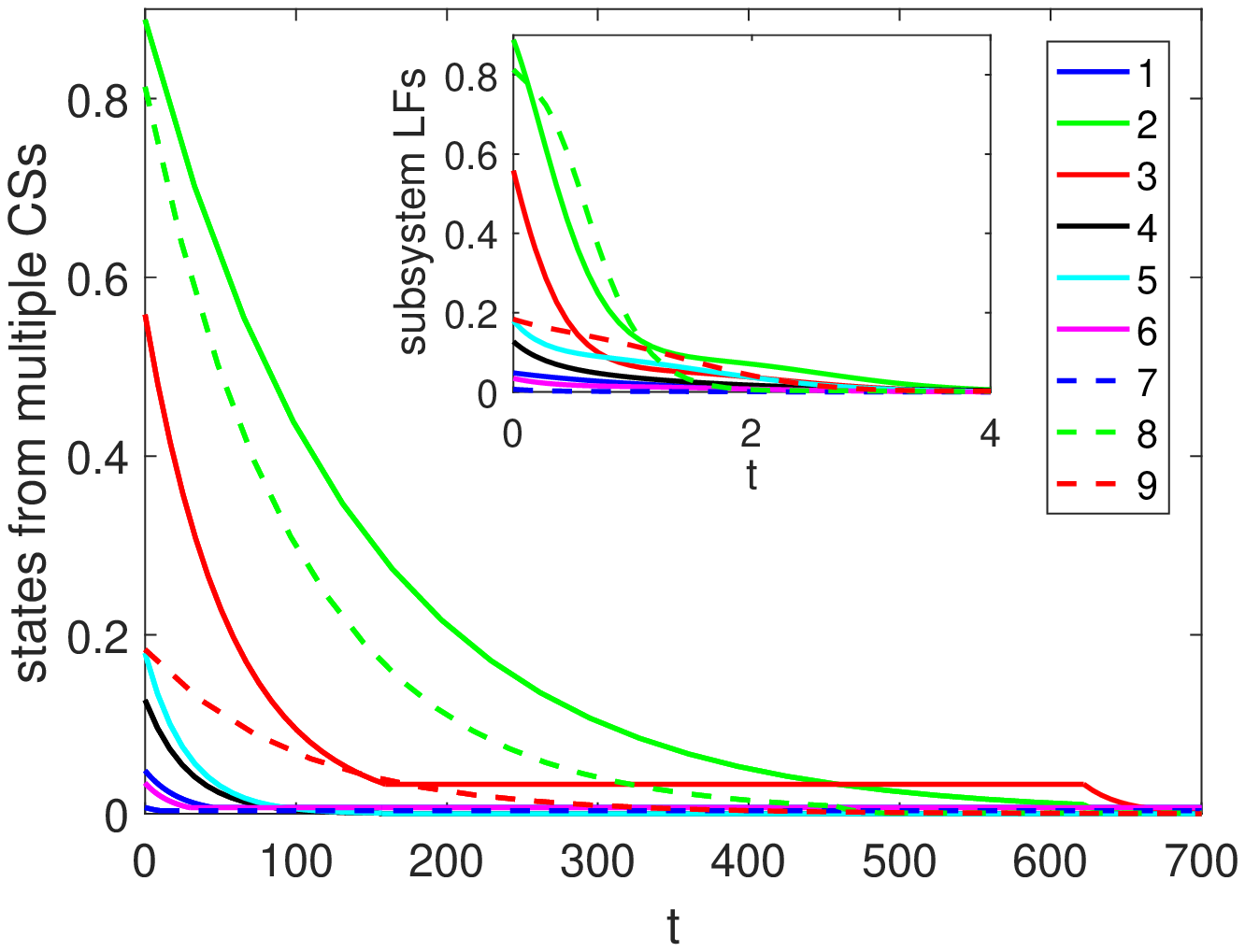}\label{F:LF_CS_quad_64}
}\hspace{0.001in}
\subfigure[Quartic LFs analysis]{
\includegraphics[scale=0.36]{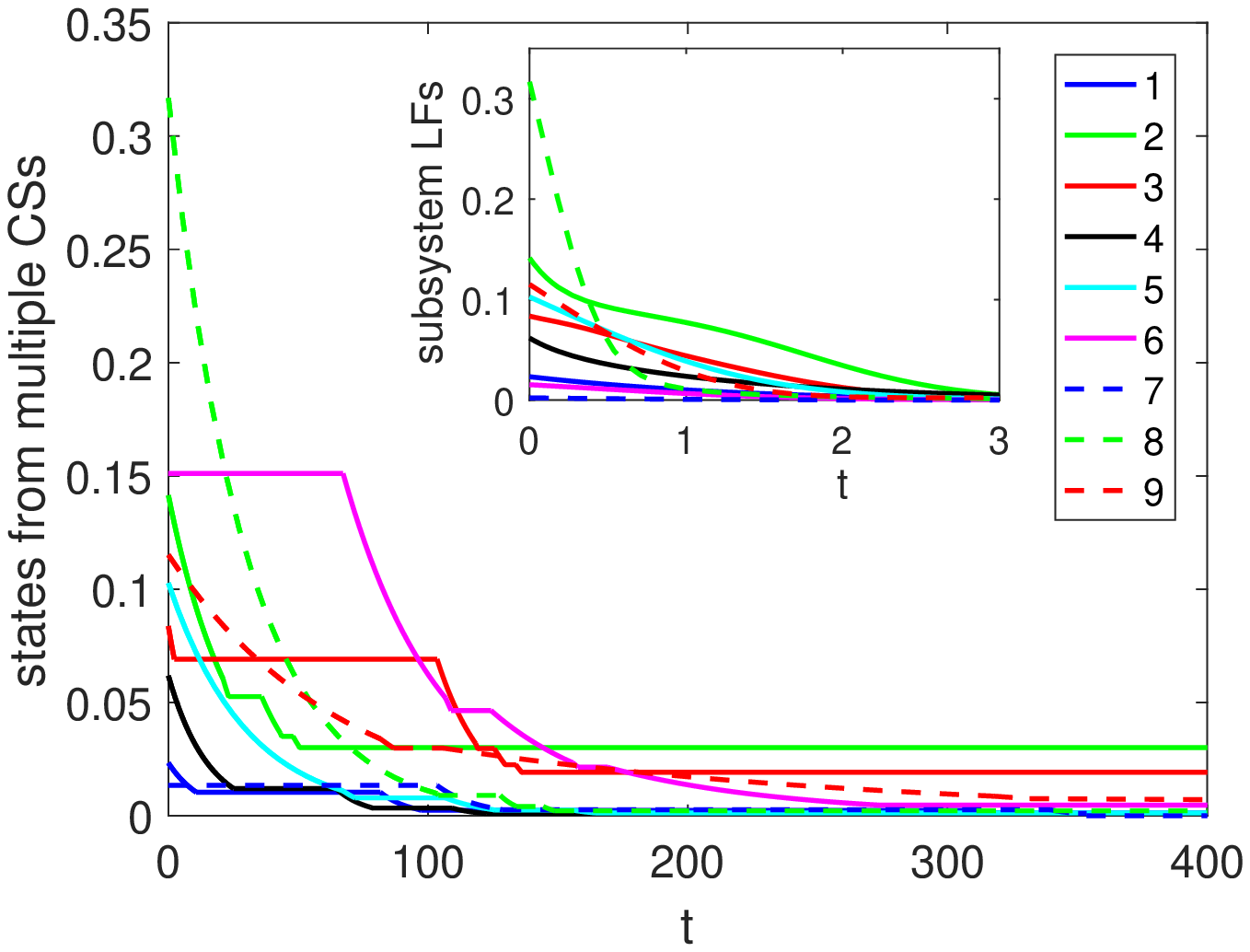}\label{F:LF_CS_quart_64}
}\caption[]{Analysis of (a) a given stable initial condition using (b) quadratic and (c) quartic LFs, and multiple CSs.}\label{F:example_64}
\end{figure*} 
\begin{remark}
Note that the algorithm estimates the exponential convergence rates rather conservatively. While this issue may be resolved by constraining the row-sum values to be less than some chosen negative number, it will likely delay the convergence of the stability algorithm.
\end{remark}
   
\section{Conclusion}\label{S:concl}

In this article we have used vector Lyapunov functions to design an iterative, distributed and parallel algorithm to certify exponential stability of a nonlinear network, under initial disturbances. The algorithm requires one-time computation of Lyapunov functions of the isolated subsystems, and minimal real-time communications between the neighboring subsystems. It is shown that the proposed SOS-based direct approach towards computation of the \textit{single comparison system} leads to less conservative certificates than the traditional methods. Further, a generalization has been proposed via \textit{multiple comparison systems}, which has been found to yield improved results compared to the \textit{single comparison system} approach. Using the pairwise interactions, a parallel implementation is also proposed, which enables the algorithm to scale up smoothly with the size of the largest neighborhood. The distributed stability analysis algorithm has been tested on an arbitrary network of nine Van der Pol systems, using vectors of  quadratic and quartic Lyapunov functions.  It is easy to visualize a multi-agent distributed coordinated control framework where each subsystem (`agent') will coordinate with its neighbors to design \textit{local} control policies to stabilize the system under disturbances. Finally, it would be interesting to explore the applicability of the proposed algorithm on real-world problems, such as the transient stability analysis of large-scale structure-preserving power system networks.

\begin{ack}                               
The authors are grateful to the U.S. Department of Energy for supporting the research presented in here, through LANL/LDRD program,  
\end{ack}

\bibliographystyle{plain}        
\bibliography{autosam,references}           

\appendix
\section{Model Description}\label{A:model}
The subsystem dynamics in the original state variables (i.e. before shifting), $\tilde{x}_i\!=\![\tilde{x}_{i,1} ~\tilde{x}_{i,2}]^T$, is given by $\dot{x}_i\!=\!\tilde{f}_i(\tilde{x}_i)+\sum_{j\in\mathcal{N}_i\backslash\left\lbrace i\right\rbrace}\tilde{g}_{ij}(\tilde{x}_i,\tilde{x}_j)~\forall i\,,$ where $\tilde{f}_i(\tilde{x}_i)\!=\![\tilde{x}_{i,2}\,,~\mu_i\,\tilde{x}_{i,2}\,(1\!-\!\tilde{x}_{i,1}^2) - \tilde{x}_{i,1}]^T$ and $\tilde{g}_{ij}(\tilde{x}_i,\tilde{x}_j)\!=\!c_{ij} + \tilde{\beta}_{ij}^{(1)}(x_{i,1}\!-\!x_{j,2}) + {\beta}_{ij}^{(2)}\,\tilde{x}_{j,2}\,\tilde{x}_{i,1}\,,$ where, $\mu_i\!\in\![-3,-1]$, $c_{ij}\!\in\![-0.2,0.2]$, $\tilde{\beta}_{ij}^{(1)}\!\in\![-0.1,0.1]$ and ${\beta}_{ij}^{(2)}\!\in\![-0.1,0.1]$ are chosen randomly. By shifting the equilibrium point $\tilde{x}_i^* \!=\!\left\lbrace \sum_{j}c_{ij}/(1\!-\!\sum_j\tilde{\beta}_{ij}^{(1)}),\,0\right\rbrace\,\forall i\,,$ to origin with $x_i\!=\!\tilde{x}_i-\tilde{x}_i^*~\forall i\,$, we obtain \eqref{E:VdP}, where $c_i^{(2)}\!=\!2\sum_{j}c_{ij}/(1\!-\!\sum_j\tilde{\beta}_{ij}^{(1)})\,$, and $\beta_{ij}^{(1)}\!=0.5\,\beta_{ij}^{(2)}c_i^{(2)}\!-\!\tilde{\beta}_{ij}^{(1)}$.

\end{document}